\documentclass[sn-mathphys, Numbered]{sn-jnl}
\usepackage{graphicx}%
\usepackage{multirow}%
\usepackage{amsmath,amssymb,amsfonts}%
\usepackage{amsthm}%
\usepackage{mathrsfs}%
\usepackage[title]{appendix}%
\usepackage{xcolor}%
\usepackage{textcomp}%
\usepackage{manyfoot}%
\usepackage{booktabs}%
\usepackage{algorithm}%
\usepackage{algorithmicx}%
\usepackage{algpseudocode}%
\usepackage{listings}%

\usepackage{amsfonts}
\usepackage{amssymb}
\usepackage{graphicx}
\usepackage{epsfig}
\usepackage{xcolor}
\usepackage{amsmath}
\usepackage{amsthm}

\newtheorem{thm}{Theorem}

\begin{document}
	
	\title{Ellipses and Polynomial-to-Polynomial Mapping of Weighted Szeg\H{o} Projections}
	%\begin{frontmatter}

		\author*[1]{\fnm{Alan R.} \sur{Legg}} \email{leggar01@pfw.edu}
	\affil*[1]{\orgdiv{Department of Mathematical Sciences}, \orgname{Purdue University Fort Wayne}, \orgaddress{ \postcode{46805}, \state{IN}, \country{USA}}}

		\abstract{We take a look at weighted Szeg\H{o} projections on ellipses and ellipsoids in light of some known results of real and complex potential theory.  We show that on planar ellipses there is a weighted Szeg\H{o} projection taking polynomials to polynomials without increasing degree.}
	
	\maketitle
	%\end{frontmatter}

	\section{Introduction}
	The goal of this note is to establish polynomial-to-polynomial mapping of certain weighted Szeg\H{o} projections on ellipses. Motivation along the way comes from the case of complex ellipsoids in several complex dimensions. Indeed in several complex variables ellipsoids are the setting of an interplay among the Szeg\H{o} projection, the Bergman projection, and harmonic functions. The Szeg\H{o} and Bergman projections coincide when acting on harmonic functions on a ball that are smooth up to the boundary of the ball. Burbea conjectured that this coincidence should characterize balls and Ligocka showed he was right \cite{L}. Even more, she showed that if any smoothly positively weighted Szeg\H{o} projection coincides with the Bergman projection on harmonic functions smooth up to the boundary of a smooth bounded domain, then the domain must be an ellipsoid of a particular form (see Theorem \ref{Ligo} below) \cite{L} . 

Ellipsoids also have a venerable place in the general history of potential theory. For an excellent exposition see Khavinson and Lundberg \cite{KL}; I would like to point out some of the ideas that will be helpful.

The first is the notion of a Fischer operator. Since an ellipsoid has a degree-two polynomial defining function $r$, and since the Laplacian $\Delta$ is a differential operator of order two, we can combine multiplications by $r$ with applications of $\Delta$ to produce operators that take polynomial inputs to polynomial outputs of the same or lesser degree. See for example \cite{S, R, R2, F}.  This seemingly innocent idea leads to a stunningly beautiful proof of the fact that on any ellipsoid the Dirichlet problem for the Laplacian posed with polynomial data on the boundary has polynomial solution of the same or lesser degree (see Theorem \ref{Ell} below) \cite{F, KL}. 

Meanwhile the Bergman projection on ellipsoids in several complex variables also takes polynomials to polynomials without increasing degree \cite{Yo}. And in the plane there is a connection to the Khavinson-Shapiro conjecture. This conjecture (from \cite{KS}) contends that ellipsoids are the \textit{only} domains on which the Dirichlet problem for the Laplacian with any polynomial data posed on the boundary has polynomial solution. Advances have been made toward this conjecture, some using Fischer operators, e.g. \cite{LR, R2, R}. For smooth bounded domains in the plane the conjecture is unaffected by swapping the Dirichlet problem solution with the Bergman projection \cite{Yo}.

I would like to bring the Szeg\H{o} projection into this circle of ideas. I will note that on complex ellipsoids in several complex variables, weighted Szeg\H{o} projections map polynomials to polynomials without increasing degree. This comes quickly from Ligocka's result and others mentioned above, but in the planar case the argument is applicable only to discs. So for planar ellipses I will use a Fischer-type operator to show that there are weighted Szeg\H{o} projections that map polynomials to polynomials without increasing degree. At the end I will pose some questions about relating the Szeg\H{o} projection to the Khavinson-Shapiro conjecture.

	\section{Background}
	\subsection{Complex Notation}
	For points $z$ in the complex plane $\mathbb{C}$ we will write $z=x+iy$ where $x,y$ are the real and imaginary parts of $z$ respectively. We will associate $\mathbb{C}$ and $\mathbb{R}^2$ in the usual way. Similarly we will associate $\mathbb{C}^n$ with $\mathbb{R}^{2n}$, and for point $z:=(z_1, z_2, \cdots, z_n) \in \mathbb{C}^n$ we will for each component $z_j$ write $z_j=x_j+iy_j$ in real and imaginary parts. For a multi-index $\alpha=(\alpha_1, \alpha_2, \cdots, \alpha_n)$ of non-negative integers we define 
	\[x^\alpha := \prod_{j=1}^n x_j^{\alpha_j}\]
	and similarly for $y^\alpha$ and for the complex-valued $z^\alpha$.

For a complex number $z \in \mathbb{C}$ we write $\bar{z}:=x-iy$ for the complex conjugate, and then $2x=z+\bar{z}$, $2iy=z-\bar{z}$ and $|z|:=x^2+y^2=z\bar{z}.$ Letting $f$ be a complex-valued function on a domain $\Omega \subseteq \mathbb{C}$ that is continuously differentiable in $x,y$, we define holomorphic differentiation $\partial/\partial z$ and antiholomorphic differentiation ${\partial}/ \partial \bar{z}$ by
\[\frac{\partial f}{\partial z} := \frac{1}{2} \; \big{(}\,\frac{\partial f}{\partial x}-i \, \frac{\partial f}{\partial y}\, \big{)} \]
\[\frac{{\partial} f}{{\partial} \bar{z}} := \frac{1}{2} \; \big{(}\, \frac{\partial f}{\partial x} + i \, \frac{\partial f}{\partial y}\, \big{)} \]
Importantly $4\partial^2/\partial z \partial \bar{z}=\Delta$, where $\Delta$ is the Laplacian. Notice also that when $f$ is real-valued $2\, \partial f/ \partial \bar{z}$ is the gradient $\nabla f$ encoded as a complex number. 

In the plane $\partial f$ and $\bar{\partial}f$ will sometimes serve as abbreviations for $\partial f/ \partial z$ and $\partial f / \partial \bar{z}$. (In several variables these symbols would connote exterior derivatives of differential forms, but we do not intend for that interpretation here.)

We say $f$ is a \textit{holomorphic function} on $\Omega$ if $\bar{\partial} f =0$ on all of $\Omega$. If $\partial f=0$ we say $f$ is \textit{antiholomorphic}. If $\Delta f=0$ we say $f$ is \textit{harmonic}. For a function in several complex variables a holomorphic function is one that is holomorphic in each complex variable separately.
	
	By a \textit{polynomial} on $\mathbb{C}^n$ we will mean a complex-valued polynomial of the real and imaginary parts. So a polynomial of degree at most $N$ on $\mathbb{C}^n$ (or $\mathbb{R}^{2n})$ will be one of the form $\sum_{|\alpha|+|\beta| \leq N}c_{\alpha,\beta} x^\alpha y^\beta$ where $c_{\alpha,\beta}$ are complex constants. Here $|\cdot|$ refers to the size of a multi-index, for example $|\alpha|:=\sum_{j=1}^n \alpha_j$.  Note that by simple arithmetic we may alternately write the polynomial above in the form
	\[\sum_{|a|+|b| \leq N} d_{a,b} z^a\bar{z}^b \]
	for (possibly different) multi-indices $a,b$ and complex constants $d_{a,b}$. By a \textit{holomorphic} polynomial we mean a polynomial of the form
	$\sum_{|\alpha| \leq N} c_\alpha z^\alpha.$ Of course a holomorphic polynomial is a holomorphic function.
	
	\subsection{The Bergman and Szeg\H{o} Projections}
	For $\Omega \subseteq \mathbb{C}^n$ a bounded domain the \textit{Bergman space} $\mathcal{A}^2(\Omega)$ is the space of holomorphic $L^2$  complex-valued functions on $\Omega$.  The \textit{Hardy space} $H^2(\partial \Omega)$ is the completion of the space of restrictions to the boundary of holomorphic functions on $\Omega$ that are continuous up to the boundary, with respect to the inner product $\langle f, g \rangle :=\int_{\partial \Omega} f\bar{g}ds$ where $ds$ is surface measure along the boundary $\partial \Omega$ . (For precise definitions see e.g. \cite{Kr} Chapter 1, and \cite{BellBook} Chapter 6.) We can define orthogonal projections to these spaces. The \textit{Bergman projection} $B$ is the orthogonal projection from $L^2(\Omega)$ to the Bergman space. The \textit{Szeg\H{o}} projection $S$ is the orthogonal projection from $L^2(\partial \Omega)$ to the Hardy space. We will view $Sf$ as a function on $\Omega$ by associating it with the holomorphic function on $\Omega$ whose boundary values are $Sf$, when this is possible. For an excellent review of all these notions in the plane see \cite{BellBook}.
	
	In this note we consider bounded $\mathcal{C}^\infty$ smooth domains and let $\mathcal{A}^\infty(\Omega)$ be the space of holomorphic functions on $\Omega$ that extend $\mathcal{C}^\infty$ smoothly to the boundary. In the planar case the smooth members of the orthogonal complement $(H^2(\partial \Omega))^\perp$ of the Hardy space are exactly the boundary functions $\overline{H(z)}\overline{T(z)}$, where $H \in \mathcal{A}^\infty(\Omega)$ and $T$ is the unit tangent to the boundary at $z$.  Also the arclength measure $ds$ on the boundary of $\Omega$ is equal to $\overline{T(z)}dz$ (\cite{BellBook}, Chapters 2,4).
	
	Let $\omega$ be a positive $\mathcal{C}^\infty$ smooth weight function on the boundary of a smooth bounded domain $\Omega$ in the plane. Let $H^{2}_\omega(\partial \Omega)$ be the correspondingly weighted Hardy space, where the inner product of functions $f,g$ on the boundary of $\Omega$ is $\langle f, g \rangle := \int_{\partial \Omega} f \bar{g} \omega ds$. The smooth members of $(H^2_\omega(\partial \Omega))^\perp$ include boundary values $\overline{H} \overline{T} \omega^{-1}$ where $H \in \mathcal{A}^\infty(\Omega).$
	
	We say a smooth real-valued function $r$ on $\mathbb{C}$ is a defining function for $\Omega$ if $\Omega=\{z: r(z)<0\}$, $\partial \Omega = \{z: r(z)=0\}$, and $\nabla r$ is nonvanishing on $\partial \Omega$.  In this case $\bar{\partial} r$ points along the outward normal at every point of $\partial \Omega$. Thus the unit tangent along the boundary is
	\[T(z)=i \, \frac{\bar{\partial} r}{| \bar{\partial} r|}. \]
	
	 \subsection{Existing Theorems}
	  Here are formal statements of three results that have already been mentioned above.
	 
	 \begin{thm} 
	 	\label{Ligo}
	 	(Ligocka \cite{L})  Let $\Omega \subseteq \mathbb{C}^n$ be a bounded $\mathcal{C}^\infty$ smooth domain and assume there is a smooth weight $\omega$ on the boundary of $\Omega$ such that $S_\omega h = Bh$ for all $h$ harmonic and smooth up to the boundary of $\Omega$, where $B$ is the Bergman projection and $S_\omega$ is the weighted Szeg\H{o} projection. Then $\Omega$ must be an ellipsoid as follows: for some $d>0$, positive symmetric $n \times n$ matrix $[a_{ij}]$ and point $(\zeta_1, \zeta_2, \cdots, \zeta_n) \in \mathbb{C}^n$,
	 	\[\Omega =\{(z_1,z_2, \cdots, z_n) \in \mathbb{C}^n: \quad \sum_{i,j=1}^n a_{ij}(z_i-\zeta_i)(\bar{z}_j-\bar{\zeta}_j) < d\}.\] 
\end{thm}

\textbf{Remark.} Ligocka also shows what the weight $\omega$ must be and deduces that the conclusion of the theorem characterizes ellipsoids of the stated form.

Considering $[a_{ij}]$ diagonal note that in several variables the so-called \textit{complex ellipsoids}, defined by $\sum_{j} a_{jj}|z_j-\zeta_j|^2 < d$, are represented. However,  we may also wonder about the {\textit{real ellipsoids}} \{$z \in \mathbb{C}^n: \;\sum_{j=1}^n (a_jx_j^2+b_j y_j^2) < d \}$ with $d>0$ and each $a_j,b_j>0$.

Not all real ellipsoids are included in Ligocka's result. You can convince yourself by restricting to the complex plane $n=1$. In that case the theorem just offers us discs. A result about the Bergman projection that does apply to all ellipsoids in $\mathbb{C}^n$ is the following.

\begin{thm}
	\label{Yothm}
	 (See \cite{Yo}) Let $\Omega$ be any ellipsoid in $\mathbb{C}^n$ (even any real ellipsoid in the terminology introduced above). Let $B$ be the Bergman projection on $\Omega$, and define $P_N$ as the space of all complex-valued polynomials (not necessarily holomorphic) of degree at most $N$. Let $HP_N$ be the space of all holomorphic polynomials of degree at most $N$. Then $B(P_N)=HP_N$.
\end{thm}
Recall that on a smooth bounded domain $\Omega$, given any $f$ continuous up to the boundary of $\Omega$, we can find a harmonic function that is also continuous up to the boundary, with the same boundary values as $f$. We call that harmonic function $\mathbb{E}f$, the \textit{Dirichlet solution} or {\textit{harmonic extension}} of the boundary values of $f$. We call $\mathbb{E}$ the harmonic extension operator or the Dirichlet solution operator.
\begin{thm}(Well-Known, see e.g. \cite{KL}) 
	\label{Ell}
	Let $\Omega \subseteq \mathbb{R}^n$ be any ellipsoid. Let $\mathbb{E}$ be the Dirichlet problem solution operator on $\Omega$. Let $P_N$ be the space of real-valued polynomials on $\mathbb{R}^n$ of degree at most $N$, and let $\mathcal{H}_N$ be the space of harmonic real-valued polynomials on $\mathbb{R}^n$ of degree at most $N$. Then $\mathbb{E} (P_N)=\mathcal{H}_N.$
	\end{thm}
\section{Results}
\subsection{Complex ellipsoids}
We are going to make a connection to Ligocka's result in complex space, so we will use $\mathbb{E}$ to denote the harmonic extension operator for complex-valued functions on domains in $\mathbb{C}^n$. Combining theorems above we get the following attractive result.

\begin{thm} Let $\Omega \subseteq \mathbb{C}^n$ be an ellipsoid as in the conclusion of Theorem \ref{Ligo} above. Let $P_N$ be the space of complex-valued (not necessarily holomorphic) polynomials on $\mathbb{C}^n$ of degree at most $N$, and let $HP_N$ be the space of holomorphic such polynomials of degree at most $N$.Then for some smooth positive weight $\omega$ on the boundary of $\Omega,$ letting $S_\omega$ be the correspondingly weighted Szeg\H{o} projection,  we have
	\[S_\omega (P_N)=HP_N. \]
\end{thm}
{\textbf{Remark.}} Essentially this says that on any \textit{complex} ellipsoid a weighted Szeg\H{o} projection maps polynomials to polynomials without increasing degree. Nothing is claimed about general \textit{real} ellipsoids.
\begin{proof} 
	Since polynomials are smooth up to the boundary of $\Omega$, Theorem \ref{Ligo} gives us for some weight $\omega$ 
	\[S_\omega p =S_\omega \mathbb{E}p= B \, \mathbb{E} p \]
	for any $p \in P_N$, where $\mathbb{E}$ is the harmonic extension operator and $B$ is the Bergman projection. But by Theorems \ref{Yothm} and \ref{Ell}, both $B$ and $\mathbb{E}$ map polynomials to polynomials without increasing degree. Hence $S_\omega p \in HP_N$. 
	\end{proof}
This gets us thinking about ellipses in the plane. Recall that Ligocka's theorem when applied in the plane touches only on discs. The theorem we just proved says nothing about eccentric planar ellipses. 
\subsection{Ellipses in the Plane}
 For planar ellipses we can get the result through a different path. There is remarkably a Fischer-type operator that makes everything fall beautifully into place.
\begin{thm} Let $\Omega \subseteq \mathbb{C}$ be any ellipse, let $P_N$ be the space of complex-valued (not necessarily holomorphic) polynomials of degree at most $N$ in the plane, and $HP_N$ the space of holomorphic polynomials of degree at most $N$. Then for some smooth positive weight $\omega$ on the boundary of $\Omega$, we have
	\[S_\omega (P_N) = HP_N \]
	where $S_\omega$ is the weighted Szeg\H{o} projection.  
\end{thm}

\begin{proof}
	Since $\Omega$ is an ellipse there are real numbers $h,k$ and positive real numbers $a,b$ such that the degree-two polynomial
	\[r(x,y)=\frac{(x-h)^2}{a^2}+\frac{(y-k)^2}{b^2}-1 \]
is a defining function for $\Omega$.  Define the weight $\omega$ on the boundary of $\Omega$ to be 
	\[\omega:=\frac{1}{|\bar{\partial} r|}. \]
	 Then for each $H \in \mathcal{A}^\infty(\Omega)$,  we see that the restriction to the boundary of $\overline{H} \cdot \partial r$  is a member of $(H^2_\omega(\partial \Omega))^\perp$. Consider the operator
	\[\partial r \bar{\partial} \mathbb{E}p \]
	acting on $P_N$. Since $r$ is a polynomial of degree $2$ and owing to Theorem \ref{Ell} we have
	\[\partial r \bar{\partial} \mathbb{E}: P_N \to P_N. \]
	But notice also that given any $p \in P_N$, $\mathbb{E}p$ is harmonic and so $\bar{\partial} \mathbb{E}p$ is an antiholomorphic polynomial, meaning
	\[\partial r \bar{\partial} \mathbb{E}p= ( \,\bar{\partial} \mathbb{E}p\,) \cdot \partial r \in (H_\omega^2(\partial \Omega))^\perp. \]
	From here we want to view our operator as a vector space isomorphism. To do that we need quotient spaces and we begin by finding the kernel. It is obvious that if $p$ is holomorphic or if $p$ is uniformly zero on the boundary then $\partial r \bar{\partial} \mathbb{E}p=0$. 
	
	Now assume that for some $p \in P_N$, $\partial r \bar{\partial} \mathbb{E}p=0$. This implies that $r \bar{\partial} \mathbb{E} p$ is an antiholomorphic function that vanishes on the boundary. By the maximum principle (and since $r<0$ inside $\Omega$) we conclude that $\bar{\partial} \mathbb{E} p=0.$ That means $\mathbb{E}p$ is holomorphic, which in turn means that $p$ has the same boundary values as a holomorphic polynomial (cf. Theorem \ref{Ell}). We conclude that $p=h+v$, for some $h \in HP_N$ and $v$ a polynomial that vanishes uniformly on the boundary. 
	
	Letting $V_N$ be the space of complex-valued polynomials of degree at most $N$ that vanish uniformly on the boundary of $\Omega$, we have shown that 
	\[\ker \partial r \bar{\partial} \mathbb{E}=HP_N \oplus V_N \]
	
	Now we determine the preimage of this kernel. Suppose that for a polynomial $p$, $\partial r \bar{\partial} \mathbb{E}p \in HP_N \oplus V_N$. Then $\partial r \bar{\partial} \mathbb{E}p$ has the same boundary values as a holomorphic polynomial. But above we showed that on the boundary $\partial r \bar{\partial} \mathbb{E}p \in (H_\omega^{2}(\partial \Omega))^\perp.$ We conclude that $\partial r \bar{\partial} \mathbb{E}p$ vanishes uniformly on $\partial \Omega$. Thus $\partial r \cdot \bar{\partial} \mathbb{E} p$ vanishes on the boundary and in fact $\bar{\partial} \mathbb{E}p$ does as well (since $\nabla r$ is nonvanishing there). But since $\mathbb{E}p$ is harmonic, so is $\bar{\partial} \mathbb{E} p$. Being a harmonic function vanishing everywhere on the boundary, $\bar{\partial} \mathbb{E} p=0$ throughout $\Omega$. Hence $\mathbb{E}p$ is holomorphic, so $p$ has the same boundary values as a holomorphic polynomial and $p \in HP_N \oplus V_N$. This discussion reveals that $\partial r \bar{\partial} \mathbb{E}$ descends to a linear one-to-one map $\varphi$ 
	\[\varphi: P_N/(HP_N \oplus V_N)  \to P_N/(HP_N \oplus V_N), \]
	via $\varphi [p] := [\partial r \bar{\partial} \mathbb{E}p]$.
	Since the domain and range spaces have the same dimension, the fact that $\varphi$ is one-to-one implies that $\varphi$ is also onto. 

The surjectivity of $\varphi$ is exactly what we need. Given any polynomial $f \in P_N$, let $p \in P_N$ be a polynomial such that $\varphi[p]=[f]$. Then for some $h \in HP_N$ and $v \in V_N$ we have $f=\partial r \bar{\partial} \mathbb{E}p +h + v$. Hence restricted to boundary we have the orthogonal decomposition
\[f=h+\partial r \bar{\partial} \mathbb{E}p\]
so that $S_\omega f=h \in HP_N$ and the theorem is proved.
\end{proof}

\section{Unweighted Szeg\H{o} Projection}

In light of the last section we might wonder when the $unweighted$ Szeg\H{o} projection can map polynomials to polynomials. On a disc it is evident that this happens. It seems unlikely that it should happen on any other smooth bounded domain. If for some bounded smooth domain $\Omega$ the Szeg\H{o} projection takes polynomials to polynomials without increasing degree, we can rotate and translate to make a simplification, and obtain a quite rigid result for one particular case. These curiosities will be taken as evidence to formulate some questions.

So assume that for a smooth bounded planar domain $\Omega$ the (unweighted) Szeg\H{o} projection $S$ takes polynomials to polynomials without increasing degree. Then $S\bar{z}=az+b$ for some constants $a,b$.  Assuming $|a| \neq 1$ let $V$ be the image of $\Omega$ under the biholomorphic mapping
\[f(z):=e^{-i\cdot \arg(a)/2}(z-\frac{\bar{a}b+\bar{b}}{1-|a|^2}).\]
According to the biholomorphic transformation formula for the Szeg\H{o} projection (\cite{BellBook} Chapter 12) the Szeg\H{o} projection on $V$ has the property that 
\[S \bar{z}=|a|z. \]
Interestingly, direct computation shows that this is impossible on an (eccentric) ellipse. Aside from that I do not pursue the case $a \neq 0$ any further, but if the Szeg\H{o} projection of $\bar{z}$ is constant the situation is clearer.
\begin{thm}
Let $\Omega \subseteq \mathbb{C}$ be a bounded finitely-connected domain with real analytic boundary and let $S$ be its Szeg\H{o} projection. Then $\Omega$ is a disc if and only if $S \bar{z}$ is constant.
\end{thm}
\begin{proof}
That $S \bar{z}$ is constant for a disc can be readily computed.

Conversely, if $S \bar{z}=b$ for $b$ a constant, then there is $H \in A^\infty(\Omega)$ such that on the boundary
\[\bar{z}-b=\overline{H(z)} \overline{T(z)}.\]
This reveals that $H$ extends real analytically to the boundary, so $H$ extends holomorphically to a neighborhood of every point on the boundary (\cite{BellBook} Theorem 11.2 and comments after). By rearranging algebraically and noting that $T(z)=1/\overline{T}(z)$ we have
on the boundary
\[\overline{T(z)}=\frac{H(z)}{z-\bar{b}}.\] 
If $H$ were to vanish at $\bar{b}$ with $\bar{b}$ in the closure of $\Omega$, then by the extension noted above and the Riemann removable singularity theorem $H(z)/(z-\bar{b})$ would be holomorphic on $\Omega$ and continuous up to the boundary. Similarly, $H(z)/(z-\bar{b})$ would be holomorphic and continuous up to the boundary if $\bar{b}$ were outside the closure of $\Omega$. In either case by Cauchy's Theorem we would have
\[\int_{\partial \Omega} 1 ds= \int_{\partial \Omega} T(z) \overline{T}(z) ds= \int_{\partial \Omega} \bar{T}(z) dz= \]
\[\int_{\partial \Omega} \frac{H(z)}{z-b} =0. \]
In other words the arclength of the boundary is $0$, which is not tenable. Thus $\bar{b} \in \Omega$ and $H(z)/(z-\bar{b})$ has a simple pole at $\bar{b}$. Manipulating as we just did, replace $1$ by any $f \in \mathcal{A}^\infty(\Omega)$ to see that
\[\int_{\partial \Omega} f(z)ds= \int_{\partial \Omega} \frac{f(z)H(z)}{z-\bar{b}}dz= \]
\[2\pi i H(\bar{b}) \cdot f(\bar{b}) \]
by the Cauchy integral formula. But $H$ is independent of the choice of $f$ so we have shown that $\Omega$ is a one-point arclength quadrature domain (cf. \cite{BellBook} Chapter 23). Thus $\Omega$ is a disc (\cite{BGS}, \cite{G} Example 3.2).
\end{proof}
\textbf{Remark.} The smoothness and topological requirements have by no means been optimized here. In fact when $\overline{T}$ extends meromorphically throughout a domain, the domain must be an arclength quadrature domain under much more generous conditions \cite{BGS, G}. Being a one-point arclength quadrature domain can also be expressed in terms of the Szeg\H{o} \textit{kernel}. This viewpoint also leads to quite rigid results for disc-like behavior, even with minimal regularity assumptions \cite{DTZ}.

\section{Some questions}

The aspects of the Szeg\H{o} and weighted Szeg\H{o} projections we have touched on are reminiscent of the Khavinson-Shapiro conjecture. It would be interesting if in the plane the Szeg\H{o} projection were as closely related to it as the Bergman projection.  

\textbf{Question 1.} If the unweighted Szeg\H{o} projection of a bounded domain in $\mathbb{C}$ maps polynomials to polynomials, must the domain be a disc?

\textbf{Question 2.} If for a bounded planar domain with smooth boundary there exists a positive smooth boundary weight so that the weighted Szeg\H{o} projection takes polynomials to polynomials, must the domain be an ellipse?

\textbf{Question 3.} In several complex variables, characterize the domains on which a weighted Szeg\H{o} projection maps polynomials to polynomials.

\textbf{Question 4.} If the Bergman projection or Dirichlet solution operator takes polynomials to polynomials, must some weighted Szeg\H{o} projection do so as well, \textit{a priori}?

%\section*{Acknowledgements}
\bibliography{Szego}

\end{document}